\documentclass[12pt]{amsart}

\usepackage{a4wide}
\usepackage{amsfonts,amsmath,latexsym,amssymb,amsthm}
\usepackage{mathrsfs}
\usepackage{hyperref}
\usepackage{xcolor}

\newtheorem{theorem}{Theorem}[section]
\newtheorem{lemma}[theorem]{Lemma}
\newtheorem{proposition}[theorem]{Proposition}
\newtheorem{corollary}[theorem]{Corollary}

\theoremstyle{definition}

\numberwithin{equation}{section}

\newcommand{\Oseen}{\mathcal{O}}

\newcommand{\IR}{\mathbb{R}}

\newcommand{\IC}{\mathbb{C}}
\newcommand{\IZ}{\mathbb{Z}}
\newcommand{\IN}{\mathbb{N}}
\newcommand{\IQ}{\mathbb{Q}}
\newcommand{\Qbar}{\overline{\IQ}}

\newcommand{\B}{\mathscr{B}}
\newcommand{\D}{D_K}
\newcommand{\del}{\eta_{\ell}(K)}

\newcommand{\p}{\mathfrak{p}}
\newcommand{\SX}{S_X}
\newcommand{\J}{\mathscr{I}}

\newcommand{\ppp}{\mathfrak{p}}
\newcommand{\qqq}{\mathfrak{q}}
\newcommand{\vx}{\mathbf{x}}

\newcommand{\Ndel}{N_{\eta_{\ell}}}
\newcommand{\Ndisc}{N_D}
\newcommand{\aaa}{\mathfrak{a}}
\newcommand{\Sab}{S_{4}(a,b)}

\newcommand{\ce}{c_1}
\newcommand{\cz}{c_2}
\newcommand{\cd}{c_3}
\newcommand{\cv}{c_4}
\newcommand{\cf}{c_5}

\DeclareMathOperator{\Cl}{Cl}

\usepackage{color}
\definecolor{dblackcolor}{rgb}{0.0,0.0,0.0}
\definecolor{dbluecolor}{rgb}{0.01,0.02,0.7}
\definecolor{dgreencolor}{rgb}{0.2,0.4,0.0}
\definecolor{dgraycolor}{rgb}{0.30,0.3,0.30}

\begin{document}

\title{Averages and higher moments for the $\ell$-torsion in class groups}

\author{Christopher Frei}
\address{TU Graz\\
Institute of Analysis and Number Theory\\
Steyrergasse 30/II, 8010 Graz\\
Austria}
\email{frei@math.tugraz.at}

\author{Martin Widmer}
\address{Department of Mathematics\\ 
Royal Holloway, University of London\\ 
TW20 0EX Egham\\ 
UK}
\email{martin.widmer@rhul.ac.uk}

\date{6 November 2020}

\dedicatory{}


\begin{abstract}
  We prove upper bounds for the average size of the $\ell$-torsion $\Cl_K[\ell]$ of the class group of $K$, as $K$ runs through certain natural families of number fields and $\ell$ is a positive integer. We refine a key argument, used in almost all results of this type, which links upper bounds for $\Cl_K[\ell]$ to the existence of many primes splitting completely in $K$ that are small compared to the discriminant of $K$. Our improvements are achieved through the introduction of a new family of specialised invariants of number fields to replace the discriminant in this argument, in conjunction with new counting results for these invariants. This leads to significantly improved upper bounds for the average and sometimes even higher moments of $\Cl_K[\ell]$ for many families of number fields $K$ considered in the literature, for example, for the families of all degree-$d$-fields for $d\in\{2,3,4,5\}$ (and non-$D_4$ if $d=4$). As an application of the case $d=2$ we obtain the best upper bounds for the number of $D_p$-fields of bounded discriminant, for primes $p>3$. 
\end{abstract}

\maketitle

\section{Introduction}\label{introductionchapter3}
In this paper, we provide bounds for the average and higher moments of the size of the $\ell$-torsion $\Cl_K[\ell]=\{[\aaa]\in\Cl_K\: ;\: [\aaa]^\ell=[\Oseen_K]\}$ of the ideal class groups of number fields $K$ in certain families, for arbitrary $\ell\in\IN=\{1,2,3,\ldots\}$. Throughout, we order number fields $K$ by the absolute value $\D$ of their discriminant. For real-valued maps $f$ and $g$ with common domain we mean by $f(t)\ll_a g(t)$ that there exists a positive constant $C=C(a)$, depending only on $a$, such that $|f(t)|\leq C|g(t)|$ for all $t$
in the domain. Throughout this article we assume $X\geq 2$. To give the reader a quick taste of the results in this paper, here is our first theorem concerning quadratic fields.

\begin{theorem}\label{thm2}
Let $\varepsilon>0$ and $k\geq 0$ be real numbers and $\ell\in\IN$. 
As $K$ ranges over all quadratic number fields with $\D\leq X$ we have
$$\sum_{K}\#\Cl_K[\ell]^k \ll_{\ell,k,\varepsilon}
X^{\frac{k}{2}+1-\min\left\{1,\ \frac{k}{\ell+2}\right\}+\varepsilon}.$$
\end{theorem}

We now discuss an application of Theorem \ref{thm2}. For a transitive permutation group $G$ of degree $d$ and $X>0$, let $N(d,G,X)$ be the number of field extensions $K/\IQ$ of degree $d$ within a fixed algebraic closure  $\Qbar$ with $\D\leq X$ and whose normal closure has Galois group isomorphic to $G$ as a permutation group. Malle's conjecture \cite{Ma,MR2068887} predicts an asymptotic formula for $N(d,G,X)$ as $X\to\infty$. Let $p$ be an odd prime and $D_p$, $D_p(2p)$ the Dihedral group of order $2p$ and its regular permutation representation respectively. In these cases, Malle's conjecture predicts the formulas
\begin{equation*}
N(p,D_p,X)\sim c_{p}X^{\frac{2}{p-1}} \quad\text{ and }\quad N(2p,D_{p}(2p),X)\sim c_{2p}X^{\frac{1}{p}}
\end{equation*}
with positive constants $c_p, c_{2p}$ (see \cite[Example after Conjecture 1.1]{Klueners}). Currently the best upper bounds for $p>3$ are
\begin{equation*}
N(p,D_p,X)\ll_{p,\varepsilon} X^{\frac{3}{p-1}-\frac{1}{p(p-1)}+\varepsilon} \quad\text{ and }\quad N(2p,D_{p}(2p),X)\ll_{p,\varepsilon} X^{\frac{3}{2p}+\varepsilon},
\end{equation*}
the first due to Cohen and Thorne \cite[Theorem 1.1]{CohenThorne}, the second due to Kl\"uners \cite[Theorem 2.7]{Klueners}. As an immediate consequence of Kl\"uners' method and the case $k=1$ in Theorem \ref{thm2}, we can improve both bounds for all primes $p>3$. 

\begin{corollary}\label{cor1}
Let $p$ be an odd prime and $\varepsilon>0$. Then we have
\begin{alignat*}1
N(p,D_p,X)\ll_{p,\varepsilon} X^{\frac{3}{p-1}-\frac{2}{(p+2)(p-1)}+\varepsilon}\quad\text{ and }\quad N(2p,D_{p}(2p),X)\ll_{p,\varepsilon} X^{\frac{3}{2p}-\frac{1}{p(p+2)}+\varepsilon}.
\end{alignat*}
\end{corollary}
The special case $p=5$ was also considered by Larsen and Rolen \cite{LarsenRolen}.
They suggest to improve Kl\"uners' bound $X^{0.75+\varepsilon}$ \cite[Theorem 2.7]{Klueners} by counting integral points on a variety defined by a norm equation. While counting these points seems a difficult matter, their numerical experiments provide evidence that the number of these points is $\ll X^{0.698}$, which, if true,  would provide the same bound for  $N(5,D_5,X)$. The exponent ${0.7+\varepsilon}$ of Cohen and Thorne is just slightly above the latter.
Our bound is $X^{0.678...}$, and hence is slightly better than the bound suggested by the numerical experiments in \cite{LarsenRolen}.

\subsection{Background}\label{sec:background}
Let us provide here some context for Theorem \ref{thm2} and our further results. Denote the degree of the number field $K$ by $d$. Landau (see, e.g., \cite[Theorem 4.4]{Nark1980}) noticed that that the Minkowski bound implies the upper bound 
\begin{equation}\label{eq:landau}
  \#\Cl_K\ll_{d,\varepsilon}\D^{\frac{1}{2}+\varepsilon},
\end{equation}
for arbitrarily small $\varepsilon>0$. This bound is essentially sharp, and provides the ``trivial'' upper bound
for the $\ell$-torsion
\begin{equation}\label{eq:trivial}
  \#\Cl_K[\ell]\ll_{d,\varepsilon}\D^{\frac{1}{2}+\varepsilon}.
\end{equation}
However, a standard conjecture asserts that
\begin{equation}\label{conj:epsilon}
  \#\Cl_K[\ell]\ll_{d,\ell,\varepsilon}\D^\varepsilon.
\end{equation}
For some references providing motivation and background for this conjecture, we refer to \cite[Conjecture 1.1]{PTW2} and the discussion thereafter. The conjecture for $d=\ell=2$ follows from Gau\ss' genus theory. Since $\#\Cl_K[\ell^t]\leq \#\Cl_K[\ell]^t$ (consider the homomorphism $[\aaa]\rightarrow [\aaa]^\ell$ from $\Cl_K[\ell^t]$ to $\Cl_K[\ell^{t-1}]$) the conjecture also holds true for
$(d,\ell)=(2,2^t)$ and arbitrary $t\in \IN$ (see \cite[Section 7.1]{PTW2}).
Apart from that the only cases of primes $\ell$ for which improvements 
over the trivial bound have been established  are $\ell=3$ for $d\leq 4$ by pioneering work of Pierce, Helfgott, Ellenberg and Venkatesh \cite{Pierce05, Pierce06,HelfgottVenkatesh,EllVentorclass},
and more recently  the case $\ell=2$ for arbitrary $d$ by Bhargava et al. \cite{Bhargava2tor}. As noted, again in \cite[Section 7.1]{PTW2}, the improvements for $(d,\ell)=(2,3)$ hold more generally for 
$(d,\ell)=(2,3\cdot 2^t)$ using the fact that $\#\Cl_K[\ell]$  is a multiplicative function (as function of $\ell$) and then combining the bounds for $\#\Cl_K[3]$ and $\#\Cl_K[2^t]$.
Of course, this argument also applies to Theorem \ref{thm2} and shows that we could replace $\ell$ in the exponent on the right hand-side by its maximal odd divisor.

These are all cases  $(d,\ell)$ for which unconditional non-trivial upper bounds for $\#\Cl_K[\ell]$ are known.
Assuming the Riemann hypothesis for the Dedekind zeta function of the normal closure of $K$, Ellenberg and Venkatesh \cite{EllVentorclass} proved the bound
\begin{equation}\label{eq:GRH-bound}
  \#\Cl_K[\ell]\ll_{d,\ell,\varepsilon}\D^{\frac{1}{2}-\frac{1}{2\ell(d-1)}+\varepsilon}
\end{equation}
for all number fields $K$. Taking up a key idea of Michel and Soundararajan and generalising it from imaginary quadratic to arbitrary number fields they show in \cite[Lemma 2.3]{EllVentorclass} that the presence of many small primes splitting completely in $K$ leads to savings over \eqref{eq:trivial}. Together with the conditional effective version of Chebotarev's density theorem, this leads directly to the bound \eqref{eq:GRH-bound}. Small splitting primes were also used in \cite{AmoDvo03} to lower bound the exponent of the class group of CM-fields.

Subsequently, several papers took the same approach using \cite[Lemma 2.3]{EllVentorclass}, but tried to establish the existence of enough splitting primes unconditionally, at the cost of averaging or having to exclude a zero-density subset of fields in a given family. Number field counting techniques were used in combination with probabilistic methods in \cite{EllenbergPierceWood,elltor1}, the large sieve in \cite{Heath-BrownPierce}, and new effective versions of Chebotarev's density theorem in \cite{PTW, An}.

In this paper, we take a different direction by refining the core argument \cite[Lemma 2.3]{EllVentorclass} itself, see Proposition \ref{keylemma}. We render the argument in a form from which we then profit by playing two ways of counting number fields, by discriminant and by minimal height of certain generators, against each other. Possible refinements were already proposed in \cite{MR2459988}, and a first concrete step in this direction was taken by the second author in \cite{doi:10.1112/blms.12113}, leading to improvements upon \cite{EllenbergPierceWood} in some cases. Our new technique yields improvements on average in all cases of \cite{EllenbergPierceWood} and \cite{doi:10.1112/blms.12113}
(provided $\ell$ is not too small), as well as on some results in \cite{EllVentorclass,PTW,An}. For example, when $\ell>2$, the case $k=1$ in Theorem \ref{thm2}  improves the case $d=2$ of \cite[Corollary 1.1.1]{EllenbergPierceWood}, which gives an upper bound
\begin{equation}\label{eq:EPW_cor}
  \sum_{K}\#\Cl_K[\ell] \ll_{\ell,\varepsilon}
X^{\frac{3}{2}-\frac{1}{2\ell(d-1)}+\varepsilon},
\end{equation}
provided $d\in\{2,3,4,5\}$ and $\ell\geq \ell(d)$, where $\ell(2)=\ell(3)=1$, $\ell(4)=8$ and $\ell(5)=25$.

Note that control over averages is often enough for applications, as
illustrated by Corollary \ref{cor1}. Moreover, having
sufficiently good upper bounds for $k$-th moments with arbitrarily large $k$
would imply  (\ref{conj:epsilon}), as shown in \cite[Theorem 1.2]{PTW2}. Here, sufficiently good means with an exponent on $X$  independent of $k$, and valid for arbitrarily large $k$.

To our best knowledge, the only published results concerning higher moments are
those of Heath-Brown and Pierce \cite{Heath-BrownPierce} on imaginary quadratic
fields. One can easily deduce bounds for arbitrary moments from a field count and
pointwise results with small exceptional sets, such as those in
\cite{EllenbergPierceWood, PTW}: for a family $S$ of degree-$d$-fields we write
\begin{equation*}
S(X) = \{K\in S;\ \D\leq X\}.
\end{equation*}
If all but at most $O_{S,a,b,\ell}(X^{a})$ exceptional fields $K\in S(X)$ satisfy
$\#\Cl_K[\ell]\ll_{S,a,b,\ell}\D^{1/2-b}$, then
\begin{equation}\label{eq:straightforward}
  \sum_{K\in S(X)}\#\Cl_K[\ell]^k \ll_{S,a,b,\ell,\varepsilon,k} \# S(X)X^{k(1/2-b)}+X^{k/2+a+\varepsilon}.
\end{equation}
In the following, we call this the \emph{straightforward approach}. In Theorem \ref{thm2} and later
results, we give bounds for the $k$-th moment in cases where the
exceptional set is
known to be
very small. Our bounds are stronger than
\eqref{eq:straightforward} when $\ell$ is not too small in terms of the
other parameters, in particular in terms of $k$.

Last but not least we should mention that there are very few but spectacular results for the averages of $\ell$-torsion in degree-$d$-fields
that provide not only upper bounds but even asymptotics.
The case $(d,\ell)=(2,3)$ is due to Davenport-Heilbronn \cite{81} (see also the recent improvements \cite{BST13,TT13, Hou}), and $(3,2)$ due to Bhargava \cite{Bhargava05}.
In particular, these two results show that for $(d,\ell)\in \{(2,3),(3,2)\}$ the conjecture (\ref{conj:epsilon}) holds true on average.
Regarding $4$-torsion in quadratic fields  Fouvry and Kl\"uners \cite{MR2276261} have established the average value for $\#\Cl_K[4]/\#\Cl_K[2]$. Related results were obtained by Klys \cite{Kly16} for  $3$-torsion in cyclic cubic fields, and by Milovic \cite{Mil15}
for the $16$-rank in certain quadratic fields.

\subsection{Further main results}
Let us next consider the other cases of \cite{EllenbergPierceWood}, concerning degree-$d$-fields for $d\in\{3,4,5\}$ (whose normal closure does not have Galois group $D_4$ in case $d=4$). In this case, our result is as follows. Define $\delta_0(3)=2/25$, $\delta_0(4)=1/48$, and $\delta_0(5)=1/200$.

\begin{theorem}\label{thm3}
Suppose $d\in \{3,4,5\}$, and $\varepsilon>0$.
As $K$ ranges over number fields of degree $d$ with $\D\leq X$ (and non-$D_4$ in the case $d = 4$), we have
$$\sum_K\#\Cl_K[\ell] \ll_{\ell,\varepsilon}
X^{\frac{3}{2}-\min\left\{\delta_0(d), \frac{1}{(d-1)\ell+3}\right\}+\varepsilon}.$$
\end{theorem}

This improves upon Ellenberg, Pierce, and Wood's result mentioned in (\ref{eq:EPW_cor})
(for large enough $\ell$), and moreover upon \cite[Corollary 1.5]{doi:10.1112/blms.12113}. 
Assuming GRH, our method also works for general families $S$ of number fields
of fixed degree, but it loses its power if the families are too thin, that is,
if $\#S(X)=\#\{K\in S\: ;\: \D\leq X\}\ll X^{\rho}$ for
$\rho<1$ too small compared to the other parameters.

\begin{theorem}\label{thm:GRH}
Let $\varepsilon>0$, let $S$ be any family of number fields of degree $d$, and assume that
  \begin{enumerate}
  \item[(i)] the Dedekind zeta function of the normal closure of each field in $S$ satisfies the Riemann hypothesis,   
  \item[(ii)] the numbers $\rho,\ce>0$ are such that $\# S(X) \leq \ce X^\rho$ for all $X\geq 2$.
  \end{enumerate}
  Then
\begin{equation*}
    \sum_{K\in S(X)} \#\Cl_K[\ell]^k \ll_{d,\rho,\ce,\ell,k,\varepsilon}X^{\frac{k}{2}+\rho-\min\left\{\rho,\frac{\rho k}{(d-1)\ell+2}\right\}+\varepsilon}.
  \end{equation*}
\end{theorem}

For comparison, an application of the straightforward approach \eqref{eq:straightforward} with the GRH-bound
\eqref{eq:GRH-bound} from \cite{EllVentorclass} and \emph{no} exceptional fields yields
\begin{equation}\label{eq:GRH_bound_comparison}
  \sum_{K\in S(X)} \#\Cl_K[\ell]^k\ll_{d,\rho,\ce,\ell,k,\varepsilon}\# S(X) X^{\frac{k}{2}-\frac{k}{2\ell(d-1)}+\varepsilon}.
\end{equation}
Taking $\rho$ to be the smallest known value with $\# S(X)\ll_\rho X^\rho$
minimises the bound in Theorem \ref{thm:GRH} as well as the one from \eqref{eq:GRH_bound_comparison}. As long as
$\rho>\frac{1}{2}+\frac{1}{\ell(d-1)}$, our Theorem
\ref{thm:GRH} provides a stronger bound than \eqref{eq:GRH_bound_comparison}, thus giving an impression of the density of $S$ that is required for our method to yield improvements.

\subsection{Further results}
In some cases with prescribed Galois groups, our method can also work for families that are thinner than suggested above. For cyclic
extensions not covered by Theorem \ref{thm2}, we are able to improve upon
\cite{elltor1,PTW} in the case $d=3$ and, moreover, to cover higher moments
using a refinement of the  straightforward approach \eqref{eq:straightforward} based on Proposition \ref{mainprop2}.

\begin{theorem}\label{thm4}
Let $\varepsilon>0$ and $k\geq 0$ be real numbers, and $\ell\in\IN$. As $K$ ranges over cubic $A_3$-extensions of $\IQ$ with $\D\leq X$, we have
$$\sum_K\#\Cl_K[\ell]^k \ll_{\ell,k,\varepsilon}X^{\frac{k+1}{2}-\min\left\{\frac{1}{2},\frac{k}{3\ell+4}\right\}+\varepsilon}.$$
\end{theorem}

For comparison,
the straightforward approach \eqref{eq:straightforward} applied with the pointwise estimate from \cite[Theorem
7.2]{PTW} for almost all $A_3$-fields gives
$\sum_K\#\Cl_K[\ell]^k
\ll_{\ell,k,\varepsilon}X^{\frac{k+1}{2}-\min\{\frac{1}{2},\frac{k}{4\ell}\}+\varepsilon}$
upon which 
Theorem \ref{thm4} is an improvement as long as $\ell\geq 5$ and $k<2\ell$.

We can also get improvements in the case of quintic fields whose normal closure
has Galois group $D_5$, the dihedral group of order $10$. As already mentioned
in the discussion after Theorem \ref{thm2}, no asymptotics for the counting function of these fields are known. Moreover, we need to impose the same ramification restrictions as in \cite{PTW}, since we rely on results from that paper to count small splitting primes.
If the rational prime $p$ ramifies tamely in a number field $K$ whose normal closure $\tilde{K}$ has Galois group $G$ then the inertia group $I(\mathfrak{B})\subset G$
is cyclic for any prime ideal $\mathfrak{B}\subset \Oseen_{\tilde{K}}$ lying above $p$. For different prime ideals $\mathfrak{B}$ over the same rational prime $p$ these inertia groups are conjugate.  Let $n>2$ be odd and $G=D_n$, the dihedral group of symmetries of a regular $n$-gon of order $2n$, so that the conjugacy class of a reflection is the set of all reflections. Keeping this in mind we say that the ramification type of a tamely ramified prime  $p$ is generated by a reflection  if each $I(\mathfrak{B})$ is generated by a reflection.

\begin{theorem}\label{thm:D5}
Let $\varepsilon>0$ and $k\geq 0$ be real numbers, and $\ell\in\IN$. Let $S$ be the family of all quintic $D_5$-extensions of $\IQ$
for which the ramification type of $p$ is generated by a reflection in $D_5$ for every tamely ramified rational prime  $p$. Suppose moreover that $\rho,\ce>0$ are such that 
\begin{equation}\label{eq:D5_rho_cond}
    \# S(X)=\#\{K\in S;\ \D\leq X\}\leq \ce X^{\rho}
  \end{equation}
  holds for all $X\geq 2$. Then, as $K$ ranges over
  $S(X)$, we have
\begin{equation*}
\sum_K\#\Cl_K[\ell]^k \ll_{\rho,\ce,\ell,k,\varepsilon}
X^{\frac{k}{2}+\rho-\frac{12\rho k}{37\ell+24}+\varepsilon}+X^{\frac{k}{2}+\frac{1}{4}+\varepsilon}.
\end{equation*}
\end{theorem}

Note that, by \cite[Proposition 2.3]{PTW}, any $\rho$ with
\eqref{eq:D5_rho_cond} must satisfy $\rho\geq 1/2$, and Malle's conjecture
predicts that $\rho=1/2$ is indeed the optimal exponent. For comparison, with the
conjectured behaviour $\# S(X)\asymp X^{1/2}$,
the straightforward approach \eqref{eq:straightforward} applied to
\cite[Theorem 7.2]{PTW} would yield
\begin{equation}\label{eq:D5_PTW_bound}
  \sum_K\#\Cl_K[\ell]^k \ll_{\ce,\ell,k,\varepsilon} X^{\frac{k+1}{2}-\frac{k}{8\ell}+\varepsilon} + X^{\frac{k}{2}+\frac{1}{4}+\varepsilon}.
\end{equation}
Hence, with Malle's conjectured exponent $\rho=1/2$, our result provides
improvements if
$\ell > 2$ and $k<2\ell$. Taken together, Corollary \ref{cor1} and Theorem \ref{thm:D5} immediately
imply the following unconditional result with $\rho=19/28+\varepsilon$.

\begin{corollary}\label{cor:D5}
Let $\varepsilon>0$ and $k\geq 0$ be real numbers, and $\ell\in\IN$. Let $S$ be the family of all quintic $D_5$-extensions of $\IQ$
for which the ramification type of $p$ is generated by a reflection in $D_5$
for every tamely ramified rational prime  $p$. Then, as $K$ ranges over $S(X)$, we have
\begin{equation*}
\sum_K\#\Cl_K[\ell]^k \ll_{\ell,k,\varepsilon}
X^{\frac{k}{2}+\frac{19}{28}-\frac{57 k}{259\ell+168}+\varepsilon}+X^{\frac{k}{2}+\frac{1}{4}+\varepsilon}.
\end{equation*}
\end{corollary}

Compared to what one gets from the straightforward approach
\eqref{eq:straightforward} using \cite[Theorem 7.2]{PTW} and
estimating $\# S(X)$ again by Corollary \ref{cor1}, this yields an improvement
whenever
$k<24\ell/7$ and $\ell\geq 2$.

Moreover, we can get improvements for certain families of quartic $D_4$-fields studied in very recent work of An \cite{An}. For distinct and squarefree $a,b\in\IZ\smallsetminus\{0,1\}$, we denote by $\Sab$ the family of quartic number fields whose normal closure has Galois group $D_4$ and contains the biquadratic field $\IQ(\sqrt{a},\sqrt{b})$. It is shown in \cite{An} that the normal closure of every $D_4$-field contains a unique biquadratic field, and the pairs $(a,b)$ with $\Sab\neq\emptyset$ are classified in \cite[Condition 1.3]{An}.

\begin{theorem}\label{thm:D4}
  Let $\varepsilon>0$ and $k\geq 0$ be real numbers, and $\ell\in\IN$. Let $a,b\in\IZ\smallsetminus\{0,1\}$ be distinct and squarefree such that $\Sab\neq\emptyset$. Suppose moreover that $\rho,\ce>0$ are such that 
\begin{equation}\label{eq:D4_rho_cond}
    \#\{K\in \Sab;\ \D\leq X\}\leq \ce X^{\rho}
\end{equation}
holds for all $X\geq 2$. Then, as $K$ ranges over the fields in $\Sab$ with $\D\leq X$, we have
\begin{equation*}
\sum_K\#\Cl_K[\ell]^k \ll_{a,b,\rho,\ce,\ell,k,\varepsilon}X^{\frac{k}{2}+\rho-\min\left\{\rho,\frac{3\rho k}{7\ell+6}\right\}+\varepsilon}.
\end{equation*}  
\end{theorem}

By \cite[Theorem 1.2]{An}, any $\rho$ with \eqref{eq:D4_rho_cond} must satisfy
$\rho\geq 1/2$, and one might expect $\rho=1/2$ to be the correct order of
magnitude. Under the assumption that the expectated order of magnitude $\#\{K\in \Sab;\ \D\leq X\}\asymp 
X^{1/2}$ is indeed correct, the straightforward approach \eqref{eq:straightforward} using
\cite[Theorem 1.1]{An} would yield
\begin{equation*}
\sum_K\#\Cl_K[\ell]^k \ll_{a,b,\ell,k,\varepsilon}X^{\frac{k+1}{2}-\min\left\{\frac{1}{2},\frac{k}{6\ell}\right\}+\varepsilon},
\end{equation*}
upon which Theorem \ref{thm:D4} improves whenever $\ell>3$ and $k<3\ell$. 
As one can take the exponent $\rho=1$ in Theorem \ref{thm:D4} by
\cite[Corollary 1.4]{59}, we
immediately obtain the following unconditional result.

\begin{corollary}\label{cor:D4}
  Let $\varepsilon>0$ and $k\geq 0$ be real numbers, and $\ell\in\IN$. Let $a,b\in\IZ\smallsetminus\{0,1\}$ be distinct and squarefree such that $\Sab\neq\emptyset$. Then, as $K$ ranges over the fields in $\Sab$ with $\D\leq X$, we have
\begin{equation*}
\sum_K\#\Cl_K[\ell]^k \ll_{a,b,\ell,k,\varepsilon}X^{\frac{k}{2}+1-\min\left\{1,\frac{3 k}{7\ell+6}\right\}+\varepsilon}.
\end{equation*}  
\end{corollary}

This should be compared to what one gets from \cite[Theorem 1.1]{An} via
\eqref{eq:straightforward}, using \cite[Corollary 1.4]{59} to
estimate $\#\{K\in \Sab;\ \D\leq X\}\ll X$, which yields
\begin{equation*}
\sum_K\#\Cl_K[\ell]^k \ll_{a,b,\ell,k,\varepsilon}X^{\frac{k}{2}+1-\min\left\{1,\frac{k}{6\ell}\right\}+\varepsilon}.
\end{equation*}  

Our techniques can also provide improved average and higher moment bounds for
some results that are conditional on open conjectures. In \cite{PTW}, the assumption of GRH was replaced for certain families of number fields by other assumptions, at the price of introducing certain ramification conditions and allowing a small exceptional set. We can also improve some of these conditional results on average.

\begin{theorem}\label{thm5}
Let $\varepsilon>0$ and $k\geq 0$ be real numbers, and $\ell\in\IN$. Let $d\geq 3$ and $S$ be the family of all number fields of degree $d$ with squarefree discriminant, whose normal closure has full Galois group $S_d$ over $\IQ$. Suppose that
\begin{enumerate}
\item[(i)] the strong Artin conjecture holds for all irreducible Galois representations over $\IQ$ with image $S_d$,
\item[(ii)] the numbers $\tau<1/2+1/d$ and $\cz$ are such that for every integer $D$, there are at most $\cz D^{\tau}$ fields $K\in S$ with $\D= D$,
\item[(iii)] the numbers $\rho,\ce>0$ are such that $\#\{K\in S;\ \D\leq X\}\leq \ce X^{\rho}$ for all $X\geq 2$.
\end{enumerate}

Then, as $K$ ranges over all elements of $S$ with $\D\leq X$, we have
$$\sum_{K} \#\Cl_K[\ell]^k \ll_{d,\rho,\ce,\cz,\ell,k,\tau,\varepsilon}
X^{\frac{k}{2}+\rho-\frac{\rho k}{(d-1)\ell+2}+\varepsilon}+X^{\frac{k}{2}+\tau+\varepsilon}.
$$
\end{theorem}
The assumptions (i) and (ii) of Theorem \ref{thm5} are the same as in \cite[Theorem 13]{PTW} for $d\geq 6$. For a precise formulation of the strong Artin conjecture, see \cite[Conjecture F]{PTW}. For $d\in\{3,4,5\}$, our assumptions can be weakened as in \cite{PTW}. If $d= 3,4$, the result is unconditional if one takes $\rho=1$ (using \cite{81} and \cite{Bhargava05}) and $\tau=1/3$ or $\tau=1/2$, respectively 
(see Theorem \ref{exceptionalfields2}). If $d=5$, one still needs (i), but one can take $\rho=1$ and the upper bound for $\tau$ in (ii) can be replaced by $1$ (see Theorem \ref {exceptionalfields2}). 

Note that Bhargava, Shankar and Wang \cite{BhargavaShankarWangSquarefree} have
shown that $\rho\geq 1/2+1/d$, and Bhargava \cite{BhargavaEkedahl} conjectured
that (iii) is sharp with $\rho=1$. On the other hand, it is conjectured that
(ii) holds with $\tau=0$ (see \cite{Ellenberg2005}). Assuming these conjectured
values for $\rho$ and $\tau$ to be the right ones, the straightforward approach \eqref{eq:straightforward} applied to the bounds
from \cite[Theorem 7.2]{PTW} would yield
\begin{equation*}
  \sum_{K} \#\Cl_K[\ell]^k \ll_{d,\ell,k,\varepsilon} X^{\frac{k}{2}+1-\min\{1,\frac{k}{2\ell(d-1)}\}+\varepsilon},
\end{equation*}
upon which Theorem \ref{thm5} yields an improvement when $k<2\ell(d-1)$ and $\ell\geq 2$.

 Finally, we can also improve the conditional result of \cite{PTW} on $A_d$-extensions for all $d\geq 5$. 

\begin{theorem}\label{thm:Ad}
Let $\varepsilon>0$ and $k\geq 0$ be real numbers. Let $d\geq 5$ and $S$ be the family of all number fields of degree $d$, whose normal closure has Galois group $A_d$ over $\IQ$. Suppose that
  \begin{enumerate}
  \item[(i)] the strong Artin conjecture holds for all irreducible Galois representations over $\IQ$ with image $A_d$,
  \item[(ii)] the numbers $\rho,\ce>0$ are such that $\#\{K\in S;\ \D\leq X\}\leq \ce X^{\rho}$ for all $X\geq 2$.
  \end{enumerate}
Then, as $K$ ranges over all fields in $S$ with $\D\leq X$, we have
 \begin{equation*}
\sum_{K} \#\Cl_K[\ell]^k \ll_{d,\rho,\ce,\ell,k,\varepsilon}X^{\frac{k}{2}+\rho-\min\left\{\rho,\frac{\rho k}{(d-3/2)\ell+2}\right\}+\varepsilon}.
\end{equation*}
\end{theorem}
Here, Malle's conjecture predicts the optimal exponent $\rho=1/2$. Assuming
this conjecture to be correct, we would get from \eqref{eq:straightforward}
applied to \cite[Theorem 7.2]{PTW} the average bound
\begin{equation*}
  \sum_{K} \#\Cl_K[\ell]^k \ll_{d,\ell,k,\varepsilon} X^{\frac{k+1}{2}-\min\{\frac{1}{2},\frac{k}{2\ell(d-1)}\}+\varepsilon}.
\end{equation*}
Theorem \ref{thm:Ad} improves the latter when $\ell>4$ and $k<(d-1)\ell$.

\subsection{Plan of the paper} 
In \S\ref{sectionkeylemma}, we introduce invariants $\eta_\ell(K)$ of number fields $K$ and use them to refine the key lemma \cite[Lemma 2.3]{EllVentorclass} of Ellenberg and Venkatesh. In \S\ref{sec:framework}, we prove two general results that use the refined key lemma to deduce average and moment bounds for $\ell$-torsion from certain asymptotic counting results. In \S\ref{counting}, we provide such counting results for fields $K$ of bounded $\eta_\ell(K)$. In \S\ref{sec:badfields}, we recall results from the literature that guarantee the existence of enough small split primes. In \S\ref{proofs}, we deduce all of our theorems, and in \S\ref{Dhe} we prove Corollary \ref{cor1}. 

\section{A refined key lemma}\label{sectionkeylemma}
Let 
\begin{alignat*}1
H_K(\alpha)=\prod_{v\in M_K}\max\{1,|\alpha|_v\}^{d_v}
\end{alignat*}
be the multiplicative Weil height of $\alpha\in K$ relative to $K$. Here $M_K$ denotes the set of places of $K$, and for each place $v$ we choose the unique representative $| \cdot |_v$ 
that either extends the usual Archimedean absolute value on $\IQ$ or a usual $p$-adic absolute value on $\IQ$, and $d_v = [K_v : \IQ_v]$ denotes the local degree at $v$.

For every prime ideal $\ppp$ of $K$ lying above a rational prime $p$, we write $e(\ppp)=e(\ppp/p)$ for the ramification index and $f(\ppp)=f(\ppp/p)$ for the inertia degree of $\ppp$ over $p$. For each $\ell\in \IN$ we introduce a new invariant of number fields $K$,

\begin{equation*}
\del=\inf\left\{H_K(\alpha)\: ;\:
  \begin{aligned}
&\alpha\in K,\ \alpha\Oseen_K=(\p_1{\p_2}^{-1})^\ell,\  \text{where $\ppp_1\neq\ppp_2$ are prime}\\ &\text{ideals of $\Oseen_K$ with $e(\ppp_i)=f(\ppp_i)=1$ for $i=1,2$}
\end{aligned}
\right\}.
\end{equation*}
We will show in Lemma \ref{lem:alpha_prop} that an element $\alpha$ of this special form necessarily generates $K$, and moreover its minimal polynomial has a restricted shape. This will allow us to deduce upper bounds for the number of fields $K$ of bounded $\del$ which lead to the improved bounds in our theorems. The following proposition is a refinement of \cite[Lemma 2.3]{EllVentorclass} and central to all our improvements.

\begin{proposition}\label{keylemma}
Let $K$ be a number field of degree $d$, $\delta<1/\ell$, and $\varepsilon>0$.
Moreover, suppose that there are $M$ prime ideals $\ppp$ of $\Oseen_K$ with norm $N(\ppp)\leq \del^{\delta}$ that satisfy $e(\ppp)=f(\ppp)=1$. If $M>0$, we have
$$\#\Cl_K[\ell] \ll_{d,\delta,\varepsilon} \D^{1/2+\varepsilon}M^{-1}.$$
\end{proposition}

\begin{proof}
We may assume that $\del\geq 2$. Write $R_K$ for the regulator of $K$ and set $G:=\Cl_K/\Cl_K[\ell]$, so that $\#\Cl_K[\ell]\cdot \#G\cdot R_K=\#\Cl_K R_K\ll_{d,\varepsilon} \D^{1/2+\varepsilon}$. Hence, we need to show that $\#G \gg_{d, \varepsilon} M/R_K$. Fix a constant $c>0$ and write $R:=\lceil c R_K\rceil$. Our goal is to show that $\#G \geq M/R$, if $c$ was chosen sufficiently large in terms of only $d$ and $\delta$. Since $R_K\gg_d 1$, we may assume that $R\geq 2$. Suppose $\#G<M/R$. Then, by the pigeonhole principle, the classes $[\ppp]$ of at least $R+1$ out of our $M$ prime ideals $\ppp$ must lie in the same coset in $G$. We call these prime ideals $\ppp_1,\ldots,\ppp_{R+1}$ to obtain $[\p_{R+1}]\Cl_K[\ell]=[\p_{i}]\Cl_K[\ell]$ for all $1\leq i \leq R$, and thus find $\alpha_i\in K$ with
\begin{alignat*}1
\alpha_i\Oseen_K=(\p_i\p_{R+1}^{-1})^\ell.
\end{alignat*}
First suppose that $K$ is imaginary quadratic. We choose distinct $i$ and $j$ between $1$ and $R$ and conclude
\begin{alignat*}1
H_K(\alpha_i/\alpha_j)\leq \max\{N(\p_i),N(\p_j)\}^\ell<\del,
\end{alignat*}
which contradicts the minimality assumption in the definition of $\del$.

Now suppose that $K$ is not imaginary quadratic.
Let $l:K^*\to\IR^{q+1}$ be the classical logarithmic embedding, where $q+1$ is the number of Archimedean places of $K$. After multiplying $\alpha_i$ by a unit we can assume that $l(\alpha_i)=(d_v\log|\alpha_i|_v)_{v| \infty}\in F+(d_v)_{v| \infty}(-\infty,\infty)$,
where $F$ is a fundamental cell of the unit lattice $l(\Oseen^*)\subset \IR^{q+1}$. We take $F=[0,1)u_1+\cdots+[0,1)u_q$ where $u_1,\ldots,u_q$ is a Minkowski reduced basis of the unit lattice. Write $l(\alpha_i)=v_i+\gamma_i(d_v)_{v| \infty}$, where $v_i\in F$ and $\gamma_i\in (-\infty,\infty)$. We note that the Euclidean length $|u_i|\gg_d 1$, which follows easily from Northcott's Theorem (see, e.g., \cite[below (8.2)]{Widmer10}).
Since $F$ comes from a Minkowski reduced basis we can partition $F$ into at most $R-1$ subcells of diameter $\ll_d (R_K/R)^{1/q}\leq c^{-1/q}\leq c^{-1/d}$. Again by the pigeonhole principle, we find distinct $i$ and $j$ such that $v_i$ and $v_j$ lie in the same 
subcell and hence $|(v_i-v_j)_v|\ll_d c^{-1/d}$ for all $v| \infty$. Without loss of generality, we may assume that $\gamma_i\leq \gamma_j$. Since $|\alpha_i/\alpha_j|_v=e^{(1/d_v)(v_i-v_j)_v+(\gamma_i-\gamma_j)}$, we conclude that 
\begin{alignat*}1
  |\alpha_i/\alpha_j|_v= e^{O_d(c^{-1/d})+(\gamma_i-\gamma_j)}\leq e^{O_d(c^{-1/d})} \quad \text{holds for all } v |\infty.
\end{alignat*}
Since $\alpha\Oseen_K=(\p_i\p_j^{-1})^\ell$, this shows that
\begin{equation*}
H_K(\alpha_i/\alpha_j)\leq e^{O_d(c^{-1/d})}N(\p_j)^{\ell}\leq e^{O_d(c^{-1/d})}\del^{\ell\delta}.
\end{equation*}
Since $\ell\delta<1$ and $\del\geq 2$, we can choose $c$ large enough in terms of $d,\delta$ to ensure that $H_K(\alpha_i/\alpha_j)<\del$, contradicting the definition of $\del$. Thus, with this choice of $c$ we get $\#G\geq M/R\gg_{d,\delta} M/R_K$.
\end{proof}

\section{Framework}\label{sec:framework}

Let $d>1$ be an integer. We set
\begin{alignat}1\label{coll}
S_{\IQ,d}=\{K\subset \Qbar\: ;\: [K:\IQ]=d\}
\end{alignat}
for the collection of all number fields of degree $d$. For a subset $S\subset S_{\IQ,d}$
we set 
\begin{alignat*}1
\SX&:=\{K\in S\: ;\: X\leq \D<2X\},\\
\B_S(X;Y,M)&:=\{K\in \SX\: ;\: \text{ at most $M$ primes $p\leq Y$ split completely in $K$}\},\\
\Ndel(S,X)&:=\#\{K\in S\: ;\: \del<X\},\\
\Ndisc(S,X)&:=\#\SX.
\end{alignat*}
Throughout this section we assume that $\theta,\rho,\ce,\cd>0$ are such that for all $X\geq 2$
\begin{align}
\label{ND} \Ndisc(S,X)&\leq \ce X^\rho,\\
\label{Ndel} \Ndel(S,X)&\leq \cd X^{\theta}.
\end{align}
We can now formulate our two main propositions. They differ in their assumption on $\#\B_S(X; X^{\delta},c X^{\delta}/\log X))$.
In the first case we have an upper bound that gets worse when $\delta$ gets smaller.
This situation happens in the work  \cite{EllenbergPierceWood} based on probabilistic methods.
In the proof
we decompose the set of fields in those fields with ``small'' invariant $\del$ compared to the discriminant, those fields with ``large'' invariant $\del$ which are not bad
(i.e., they have ``sufficiently'' many small splitting primes), and those fields with ``large'' invariant $\del$ which are bad. In the first and third case, we use the trivial bound to estimate $\#\Cl_K[\ell]$, and in the second case Proposition \ref{keylemma}.

\begin{proposition}\label{mainprop1}
Suppose $S\subset S_{\IQ,d}$, $\delta_0>0$, and that \eqref{ND}, \eqref{Ndel} hold for $\theta,\rho,\ce,\cd>0$.  
Moreover, suppose for every $\delta\in (0,\delta_0]$ and $\varepsilon\in(0,1)$ there are positive $\cv(\delta,\varepsilon)$ and $\cf(\delta,\varepsilon)$  such that
\begin{alignat*}1
\#\B_S(X;X^{\delta},\cv(\delta,\varepsilon)X^{\delta}/\log X))\leq \cf(\delta,\varepsilon)X^{\rho-\delta+\varepsilon}
\end{alignat*}
holds for all $X\geq 2$. Then we have, for all $\varepsilon\in (0,1)$,
$$\sum_{K\in \SX} \#\Cl_K[\ell] \ll_{d,\ell,\theta,\rho,\ce,\cd,\delta_0,\cv(\cdot,\cdot),\cf(\cdot,\cdot),\varepsilon} 
 X^{\frac{1}{2}+\rho-\min\{\delta_0,\frac{\rho}{\ell\theta+1}\}+\varepsilon}.
$$
\end{proposition}

\begin{proof}
Let $\varepsilon\in(0,1)$.
For sake of clarity, we suppress the dependence of implicit constants in our notation and write $\ll$ instead of $\ll_{d,\ell,\theta,\rho,\ce,\cd,\delta_0,\cv(\cdot,\cdot),\cf(\cdot,\cdot),\varepsilon}$ throughout the proof. We define 
\begin{alignat*}1
\gamma_0:=\frac{\rho\ell}{\ell\theta+1}.
\end{alignat*}
Hence we have 
\begin{alignat*}1
\gamma_0\theta=\rho-\frac{\gamma_0}{\ell}.
\end{alignat*}
First let us assume that $\ell\leq \frac{1}{\theta}(\frac{\rho}{\delta_0}-1)$, and thus
\begin{alignat*}1
\gamma_0\geq \delta_0\ell.
\end{alignat*}
We decompose $\SX$ into the three subsets
\begin{alignat*}3
M_0&=\{K\in \SX\: ;\: \del\leq \D^{\delta_0\ell}\},\\
M'_{1}&=\{K\in \SX\: ;\: \del>\D^{\delta_0\ell}\}\smallsetminus \B_S(X;X^{(1-\varepsilon)\delta_0},c X^{(1-\varepsilon)\delta_0}/\log X),\\
M''_{1}&=\{K\in \SX\: ;\: \del>\D^{\delta_0\ell}\}\cap \B_S(X;X^{(1-\varepsilon)\delta_0},c X^{(1-\varepsilon)\delta_0}/\log X), 
\end{alignat*}
where $c=\cv((1-\varepsilon)\delta_0,\varepsilon)$ comes from the assumptions of the proposition. 
Using \eqref{eq:trivial},
 we get
\begin{alignat*}1
\sum_{K\in M_0} \#\Cl_K[\ell] \ll \sum_{K\in M_0} \D^{\frac{1}{2}+\varepsilon}\leq \#M_0\cdot (2X)^{\frac{1}{2}+\varepsilon}.
\end{alignat*}
Since $\#M_0\leq\Ndel(S,(2X)^{\delta_0\ell})\ll X^{\delta_0\ell\theta}$ 
and $\delta_0\ell\leq \gamma_0$ we conclude
\begin{alignat*}1
\sum_{K\in M_0} \#\Cl_K[\ell]\ll X^{\frac{1}{2}+\gamma_0\theta+\varepsilon}\leq X^{\frac{1}{2}+\rho-\delta_0+\varepsilon}.
\end{alignat*}
Since by assumption $\#M''_1\ll X^{\rho-(1-\varepsilon)\delta_0+\varepsilon}$, we find similarly
\begin{alignat*}1
\sum_{K\in M''_1} \#\Cl_K[\ell]\ll X^{\frac{1}{2}+\rho-\delta_0+(2+\delta_0)\varepsilon}.
\end{alignat*}
For the sum over $M'_{1}$ we use Proposition \ref{keylemma}, with the valid choice $M=c X^{(1-\varepsilon)\delta_0}/\log X$, and then bound $ \#M'_1$ by  (\ref{ND})
to conclude that
\begin{alignat*}1
\sum_{K\in M'_{1}} \#\Cl_K[\ell] &\ll \sum_{K\in M'_{1}} \D^{\frac{1}{2}-(1-\varepsilon)\delta_0+2\varepsilon} \leq \#M'_1\cdot (2X)^{\frac{1}{2}-\delta_0+(2+\delta_0)\varepsilon}\ll X^{\frac{1}{2}+\rho-\delta_0+(2+\delta_0)\varepsilon}.
\end{alignat*}
This proves the proposition when $\ell\leq \frac{1}{\theta}(\frac{\rho}{\delta_0}-1)$. Now let us assume that $\ell>\frac{1}{\theta}(\frac{\rho}{\delta_0}-1)$, and thus 
\begin{alignat*}1
\gamma_0< \delta_0\ell.
\end{alignat*}
We now define $M_0, M'_1$ and $M''_1$ exactly in the same way but with $\delta_0$ replaced by $\gamma_0/\ell$. Arguing in exactly the same way as in the previous case we get 
\begin{alignat*}1
\sum_{K\in \SX} \#\Cl_K[\ell] \ll X^{\frac{1}{2}+\rho-\frac{\gamma_0}{\ell}+(2+\frac{\gamma_0}{\ell})\varepsilon}\leq
X^{\frac{1}{2}+\rho-\frac{\rho}{\ell\theta+1}+(2+\rho)\varepsilon}.
\end{alignat*}
\end{proof}

Our next main proposition applies when the bound for $\#\B_S(X; X^{\delta},c X^{\delta}/\log X))$
is uniform in $\delta$. For $d=2$ such a bound can be established  by using the large sieve, as shown in
\cite{Heath-BrownPierce}. It is a new innovation of the recent work \cite{PTW} that such uniform bounds are
also available for a much larger class of families $S$. 
In this setting it turns out beneficial to use a finer decomposition of the set of fields than just those 
fields with ``small'' invariant $\del$, and  those fields with ``large'' invariant $\del$.

\begin{proposition}\label{mainprop2}
Suppose $S\subset S_{\IQ,d}$, $\tau\geq 0$, and that \eqref{ND}, \eqref{Ndel} hold for $\theta,\rho,\ce,\cd>0$.
Moreover, suppose for every $\delta>0$ and $\varepsilon\in(0,1/\ell)$ there are positive $\cv(\delta,\varepsilon)$ and $\cf(\delta,\varepsilon)$ such that
\begin{alignat*}1
\#\B_S(X;X^{\delta},\cv(\delta,\varepsilon) X^{\delta}/\log X)\leq \cf(\delta,\varepsilon) X^{\tau+\varepsilon}
\end{alignat*}
holds for all $X\geq 2$. Then we have, for all $k\geq 0$ and $\varepsilon\in(0,1/\ell)$, 
$$\sum_{K\in \SX} \#\Cl_K[\ell]^k \ll_{d,\theta,\rho,\ce,\cd,\cv(\cdot,\cdot),\cf(\cdot,\cdot),\ell,k,\tau,\varepsilon}
X^{\frac{k}{2}+\rho-\frac{\rho k}{\ell\theta}+\varepsilon}+X^{\frac{k}{2}+\tau+\varepsilon}.
$$
\end{proposition}

\begin{proof}
Let $\varepsilon\in(0,1/\ell)$.
We decompose $\SX$ into $N+2$ subsets $M_i$,
where $N=N(\varepsilon)$ will be chosen later.
Let $0=\gamma_{-1}\leq \gamma_0\leq \gamma_1\leq \cdots\leq \gamma_N$ and set 
\begin{alignat*}3
M_{i}&=\{K\in \SX\: ;\: \D^{\gamma_{i-1}}\leq\del<\D^{\gamma_{i}}\} \qquad (0\leq i\leq N),\\
M_{N+1}&=\{K\in \SX\: ;\: \D^{\gamma_N}\leq\del\}. 
\end{alignat*}
Furthermore, for $1\leq i\leq N+1$ we decompose $M_i$ into the two sets
\begin{alignat*}3
M'_{i}&=M_i\smallsetminus \B_S(X;X^{\gamma_{i-1}(1/\ell-\varepsilon)},c'_i X^{\gamma_{i-1}(1/\ell-\varepsilon)}/\log X),\\
M''_{i}&=M_i\cap \B_S(X;X^{\gamma_{i-1}(1/\ell-\varepsilon)},c'_i X^{\gamma_{i-1}(1/\ell-\varepsilon)}/\log X),
\end{alignat*}
where $c'_i=\cv(\gamma_{i-1}(1/\ell-\varepsilon),\varepsilon)$.
Hence, we have partitioned
$\SX$ into the $1+2(N+1)$ subsets $M_0,M'_i,M''_i$ ($1\leq i\leq N+1$). Throughout this proof, we suppress the implicit constants in our notation and write $\ll$ for $\ll_{d,\theta,\rho,\ce,\cd,\cv(\cdot,\cdot),\cf(\cdot,\cdot),\ell,k,\tau,\varepsilon,\gamma_{0},\ldots,\gamma_N}$. The values of $\gamma_{0},\ldots,\gamma_N$ are fixed later in the proof depending only on the other parameters. Next we record the estimates 
\begin{alignat*}3
&\#M_0&&\leq \Ndel(S,(2X)^{\gamma_0})\ll X^{\gamma_0\theta},&\\
&\#M'_{i}&&\leq\#M_{i}\leq \Ndel(S,(2X)^{\gamma_i})\ll X^{\gamma_i\theta}  &(1\leq i\leq N),\\
&\#M'_{N+1}&&\leq\#M_{N+1}\leq \Ndisc(S,X)\ll X^{\rho},&\\
&\#M''_{i}&&\ll X^{\tau+\varepsilon} &(1\leq i\leq N+1). 
\end{alignat*}
We use \eqref{eq:trivial}
to estimate the sums over $M_0$ and $M_i''$ ($1\leq i\leq N+1$),
\begin{alignat*}1
\sum_{K\in M_0} \#\Cl_K[\ell]^k &\ll \sum_{K\in M_0} \D^{(\frac{1}{2}+\varepsilon)k}\leq \#M_0\cdot (2X)^{\frac{k}{2}+k\varepsilon}\ll X^{\frac{k}{2}+\gamma_0\theta+k\varepsilon},\\
\sum_{K\in M''_i} \#\Cl_K[\ell]^k&\ll \sum_{K\in M''_i} \D^{(\frac{1}{2}+\varepsilon)k}\leq \#M''_i\cdot (2X)^{\frac{k}{2}+k\varepsilon}
\ll X^{\frac{k}{2}+\tau+(k+1)\varepsilon}.
\end{alignat*}
From Proposition \ref{keylemma},  with the eligible choice $M=c_i' X^{\gamma_{i-1}(1/\ell-\varepsilon)}/\log X$, we conclude for $1\leq i\leq N$ that
\begin{alignat*}1
\sum_{K\in M'_{i}} \#\Cl_K[\ell]^k &\ll \sum_{K\in M'_{i}} \D^{(\frac{1}{2}-\gamma_{i-1}(\frac{1}{\ell}-\varepsilon)+2\varepsilon)k} \ll X^{\frac{k}{2}-\frac{\gamma_{i-1}k}{\ell}+\gamma_{i}\theta+k(2+\gamma_N)\varepsilon}
\end{alignat*}
and similarly
\begin{alignat*}1
\sum_{K\in M'_{N+1}} \#\Cl_K[\ell]^k &\ll X^{\frac{k}{2}+\rho-\frac{\gamma_N k}{\ell}+k(2+\gamma_N)\varepsilon}.
\end{alignat*}
For $0\leq i\leq N$, we define $Q_i=\sum_{r=0}^{i}q^r$, where $q=\frac{k}{\ell\theta}$. With these quantities in place, we proceed to choose our $\gamma_i$ as follows,
\begin{alignat*}1
\gamma_0=\gamma_0(N)=\frac{\rho\ell}{\ell\theta+kQ_N}\quad\text{ and }\quad\gamma_i=\gamma_0Q_i\leq \frac{\rho\ell}{k}\quad (1\leq 1\leq N).
\end{alignat*}
Then a quick computation shows that 
\begin{alignat*}1
\frac{k}{2}+\gamma_0\theta=\frac{k}{2}-\frac{\gamma_{i-1}k}{\ell}+\gamma_{i}\theta=\frac{k}{2}+\rho-\frac{\gamma_N k}{\ell},
\end{alignat*}
which allows us to estimate
\begin{alignat*}1
\sum_{K\in \SX} \#\Cl_K[\ell]^k \ll
X^{\frac{k}{2}+\gamma_0\theta+(2k+\rho\ell)\varepsilon}+X^{\frac{k}{2}+\tau+(k+1)\varepsilon}.
\end{alignat*}
The only task left is to choose $N=N(\varepsilon)$. We observe that
\begin{alignat*}1
\tilde{\gamma}_0:=\lim_{N\rightarrow \infty}\gamma_0(N)=\begin{cases}
 \frac{\rho}{\theta}-\frac{\rho k}{\ell \theta^2} &\text{ if }q<1,\\
 0 &\text{ if }q\geq 1.
\end{cases}
\end{alignat*}
Hence, choosing $N=N(\varepsilon)$ big enough to ensure $\gamma_0\theta\leq \tilde{\gamma}_0\theta +\varepsilon$, we conclude that
\begin{alignat*}1
\sum_{K\in \SX} \#\Cl_K[\ell]^k \ll
X^{\frac{k}{2}+\tilde{\gamma}_0\theta+(1+2k+\rho\ell)\varepsilon}+X^{\frac{k}{2}+\tau+(k+1)\varepsilon},
\end{alignat*}
which proves the proposition.
\end{proof}

\section{Counting fields of bounded $\del$}\label{counting}

For $\alpha\in \Qbar$ we write $D_\alpha\in \IZ[x]$ for the minimal polynomial of $\alpha$ over $\IZ$, i.e., the irreducible polynomial with positive leading coefficient that satisfies $D_\alpha(\alpha)=0$. Our estimates for $\Ndel(S,X)$ hinge upon the following observation.

\begin{lemma}\label{lem:alpha_prop}
  Let $\alpha\in K$ be such that $\alpha\Oseen_K=(\p_1{\p_2}^{-1})^\ell$, with distinct prime ideals $\p_1,\p_2$ of $\Oseen_K$ that satisfy $e(\p_i)=f(\p_i)=1$ for $i=1,2$. Then $K=\IQ(\alpha)$ and the minimal polynomial $D_\alpha$ has the form
\begin{alignat}1\label{minpoly}
D_\alpha=p^\ell x^d+a_1x^{d-1}+\cdots+a_{d-1}x\pm q^{\ell},
\end{alignat}
where $a_1,\ldots,a_{d-1}\in\IZ$ and $p,q$ are the primes below $\p_2$ and $\p_1$, respectively.
\end{lemma}

\begin{proof}
First, suppose $\IQ(\alpha)=F\subsetneqq K$. Let $\qqq_1$ be the prime ideal of $\Oseen_F$ below $\ppp_1$. Then $e(\ppp_1/\qqq_1)=f(\ppp_1/\qqq_1)=1$. Hence, as $[K:F]>1$, there must be another prime ideal $\ppp_1'$ of $\Oseen_K$ above $\qqq_1$. For the corresponding discrete valuations, we get $v_{\ppp_1'}(\alpha)=e(\ppp_1'/\qqq_1)v_{\qqq_1}(\alpha)=e(\ppp_1'/\qqq_1)v_{\ppp_1}(\alpha)=e(\ppp_1'/\qqq_1)\ell>0$. But there is no other prime ideal of $\Oseen_K$ at which $\alpha$ has positive valuation. Hence, $\IQ(\alpha)=K$. 
The second assertion follows immediately from the well-known formula 
$$a_0=\prod_{v\nmid \infty}\max\{1,|\alpha|_v\}^{d_v},$$ 
where $a_0$ is the leading coefficient of $D_\alpha$ and the product runs over all non-Archimedean places of $\IQ(\alpha)$. The latter formula in turn is  essentially 
a consequence of Gau\ss' Lemma
applied to $D_\alpha$ and each non-Archimedean place of the  splitting field of $D_\alpha$. 
\end{proof}

\begin{lemma}\label{countingprop}
Suppose $S\subset S_{\IQ,d}$, and $\theta=d-1+2/\ell$. Then
\begin{alignat*}1
\Ndel(S,X)\ll_{d} X^{\theta}.
\end{alignat*}
\end{lemma}

\begin{proof}
  Let $P_S$ be the set of all $\alpha\in\Qbar$ such that $\IQ(\alpha)\in S$ and $\alpha\Oseen_{\IQ(\alpha)}=(\ppp_1\ppp_2^{-1})^{\ell}$, for prime ideals $\ppp_1\neq\ppp_2$ of $\Oseen_{\IQ(\alpha)}$ with $e(\ppp_i)=f(\ppp_i)=1$ for $i=1,2$. Moreover, let
\begin{equation*}
  N_H(P_S,X):=\#\{\alpha\in P_S\: ;\: H_{\IQ(\alpha)}(\alpha)\leq X\}.
\end{equation*}
Using Lemma \ref{lem:alpha_prop}, we observe that the image of the map $\alpha\rightarrow \IQ(\alpha)$ with domain 
$$\{\alpha\in P_S\: ;\: H_{\IQ(\alpha)}(\alpha)\leq X\}$$ 
covers the set
$$\{K\in S\: ;\: \del\leq X\}.$$
Hence, we get
 \begin{alignat*}1
\Ndel(S,X)\leq N_H(P_S,X).
\end{alignat*}
Now if $\alpha\in P_S$ then, as noted in (\ref{minpoly}), the first and last coefficient of its minimal polynomial $D_\alpha$ are, up to sign, $\ell-th$ prime powers. For $\alpha$ to be counted in $N_H(P_S,X)$, we also require $H_{\IQ(\alpha)}(\alpha)\leq X$. 
Now the maximum norm of the coefficient vector of $D_\alpha$ is bounded from above by
$2^dH_{\IQ(\alpha)}(\alpha)$, and hence by $2^{d}X$.
Thus, we have at most 
$\ll_d X^{d-1+2/\ell}$ possibilities for these minimal polynomials and thus for $\alpha$.
\end{proof}

The bound on $\Ndel(S,X)$ from Lemma \ref{countingprop} suffices to deduce Theorems \ref{thm2},
\ref{thm3}, \ref{thm:GRH} and \ref{thm5}. Our other theorems involve families of number
fields with specified Galois groups $G\subsetneq S_d$. To compensate for the
relative thinness of these families, we need to show that families of
polynomials of degree $d$ with specified Galois group $G\subsetneq S_d$ are
also thin. This was done by Dietmann in \cite{MR2891158}, but
his results are not applicable to our situation as they concern monic
polynomials with no further restrictions on their coefficients, whereas we have
to deal with polynomials of the shape \eqref{minpoly}.

The idea of Dietmann's proof, to detect polynomials with Galois group $G$ through roots of appropriate
resolvents, and to control these roots via uniform bounds for integral points
on affine surfaces, applies to our situation as well. The following results,
culminating in Proposition \ref{prop:sub_G} below, modify and refine Dietmann's
proofs accordingly. Hence, we keep our notation similar to that of
\cite{MR2891158}. In particular, we will write $n$ instead of $d$ for the
degree of our polynomials.

For any field $K$ of characteristic $0$ and $n\in\IN$, we consider polynomials
\begin{equation*}
  f=x^n+a_{1}x^{n-1}+\cdots+a_n\in K[x]
\end{equation*}
with distinct roots $\alpha_1,\ldots,\alpha_n$ in an algebraic closure of $K$. Let $G\subset S_n$ be a subgroup, then the Galois resolvent from \cite[Lemma 5]{MR2891158} is defined as
\begin{equation}\label{eq:galois_resolvent}
  \phi(z; a_1,\ldots,a_n) = \prod_{\sigma\in S_n/G}\left(z-\sum_{\tau\in G}\alpha_{\sigma(\tau(1))}\alpha_{\sigma(\tau(2))}^2\cdots\alpha_{\sigma(\tau(n))}^{n}\right).
\end{equation}
It is a polynomial in $z,a_1,\ldots,a_n$ with integer coefficients that do not depend on $K$. It is monic in $z$ of degree $\#(S_n/G)$. It has a root $z\in K$ whenever the Galois group of $f$, as a subgroup of $S_n$ acting on $\alpha_1,\ldots,\alpha_n$, is contained in $G$. In case $K=\IQ$ and $a_1,\ldots,a_n\in\IZ$, this root must clearly lie in $\IZ$. Moreover, we denote by $\Delta_\phi(a_1,\ldots,a_n)\in K$ the discriminant of $\phi(z;a_1,\ldots,a_n)\in K[z]$. Again, this discriminant is a polynomial in $a_1,\ldots,a_n$ with integer coefficients independent of $K$.

\begin{lemma}\label{lem:discriminant_non_zero}
  Fix $a_n\in\IQ$, $a_n\neq 0$. Then $\Delta_\phi(a_1,\ldots,a_{n-1},a_n)$ is not identically zero as a polynomial in $a_1,\ldots,a_{n-1}$.
\end{lemma}

\begin{proof}
  This is a refinement of \cite[Lemma 7]{MR2891158}. Fix $a_n\neq 0$. Then it is enough to find $a_1,\ldots,a_{n-1}\in\IC$ such that $\Delta_\phi(a_1,\ldots,a_n)\neq 0$. For any choice of $a_1,\ldots,a_{n-1}$, it is clear from \eqref{eq:galois_resolvent} that the roots of $\phi(z;a_1,\ldots,a_n)$ are the complex numbers
  \begin{equation}\label{eq:disc_zeros}
    \sum_{\tau\in G}\alpha_{\sigma(\tau(1))}\alpha_{\sigma(\tau(2))}^2\cdots\alpha_{\sigma(\tau(n))}^{n},
  \end{equation}
where $\sigma$ ranges over a set of representatives for the cosets in $S_n/G$. All $\#(S_n/G)$ of these expressions are distinct homogeneous polynomials of degree $n(n+1)/2$ in $\IZ[\alpha_1,\ldots,\alpha_n]$. Hence, there is a non-empty Zariski-open subset of points $(\alpha_1:\cdots:\alpha_n)\in\mathbb{P}^{n-1}$ for which all the expressions in \eqref{eq:disc_zeros} are distinct. In particular, we find such $(\alpha_1:\cdots:\alpha_n)$ whose homogeneous coordinates $\alpha_i\in\IC$ are all distinct and non-zero. Picking a correctly scaled representative of such a point, we get $\alpha_1,\ldots,\alpha_n\in\IC$ that satisfy all of the previous conditions and moreover that $(-1)^n\alpha_1\cdots\alpha_n=a_n$. Let $a_1,\ldots,a_{n-1}\in\IC$ be the other coefficients of the polynomial $\prod_{i=1}^n(x-\alpha_i)$. Then, by our choice of $\alpha_1,\ldots,\alpha_n$, all zeros of $\phi(z;a_1,\ldots,a_n)$ are distinct, and hence its discriminant satisfies $\Delta_\phi(a_1,\ldots,a_n)\neq 0$.
\end{proof}

\begin{lemma}\label{lem:full_galois_Qt}
  Let $n\geq 3$ and $a_2,\ldots,a_{n-2},a_n\in\IZ$ such that $a_n\neq 0$. Then the polynomial 
  \begin{equation*}
    x^n+a_1x^{n-1}+\cdots+a_{n-2}x^2+tx+a_n\in\IQ(t)[x]
  \end{equation*}
  has, for all but  $\ll_n 1$ values of $a_1\in\IZ$, the full symmetric group $S_n$ as Galois group acting on its roots in an algebraic closure of the rational function field $\IQ(t)$.
\end{lemma}

\begin{proof}
  This is similar to \cite[Lemma 2]{MR2891158}. By \cite[Satz 1]{MR0276203}, the Galois group is $S_n$ for all but finitely many values of $a_1\in\IZ$. As described in \cite[Lemma 2]{MR2891158} and the introduction of \cite{MR0302623}, the proof of \cite[Satz 1]{MR0276203} provides the upper bound $n^2$ for the number of excluded values of $a_1$.
\end{proof}

\begin{lemma}\label{lem:resolvent_irreducible}
  Let $n\geq 2$ and $a_1,\ldots,a_{n-2},a_n\in\IZ$ such that the polynomial  
  \begin{equation*}
    x^n+a_1x^{n-1}+\cdots+a_{n-2}x^2+tx+a_n\in\IQ(t)[x]
  \end{equation*}
  has Galois group $S_n$ over the rational function field $\IQ(t)$. Moreover, suppose that
  \begin{equation}\label{eq:discr_condition}
    \Delta_\phi(a_1,\ldots,a_{n-2},t,a_n)\neq 0 \text{ in } \IQ(t).
  \end{equation}
Then the polynomial $\phi(z;t)=\phi(z; a_1,\ldots,a_{n-2},t,a_n)\in\IQ[z,t]$ is irreducible over $\IQ$.
\end{lemma}

\begin{proof}
  Note that \eqref{eq:discr_condition} states that the roots of $\phi(z;t)$ in an algebraic closure of $\IQ(t)$ are all distinct. Hence, we are precisely in the situation of \cite[Lemma 6]{MR2891158}, except that we use the variable $t$ for the linear coefficient, whereas Dietmann uses $t$ for the constant coefficient. The proof of \cite[Lemma 6]{MR2891158} is agnostic of this difference and works verbatim in our case.
\end{proof}

The following result is \cite[Lemma 8]{MR2891158}, which follows from \cite[Theorem 1]{zbMATH02221889}.

\begin{lemma}\label{lem:BHB}
  Let $F\in\IZ[x_1,x_2]$ be of degree $d$ and irreducible over $\IQ$. Let $P_1,P_2\geq 1$, and
  \begin{equation*}
  T=\max_{(e_1,e_2)}\{P_1^{e_1}P_2^{e_2}\},
\end{equation*}
  where $(e_1,e_2)$ runs through all pairs for which the monomial $x_1^{e_1}x_2^{e_2}$ appears in $F$ with non-zero coefficient. Then, for $\varepsilon>0$, 
  \begin{equation*}
    \#\{\vx\in\IZ^2\: ;\: F(\vx)=0\text{ and }|x_i|\leq P_i \text{ for }i=1,2\} \ll_{d,\varepsilon}\max\{P_1,P_2\}^{\varepsilon}\exp\left(\frac{\log P_1 \log P_2}{\log T}\right).
  \end{equation*}
\end{lemma}
Note that the implicit constant depends only on the degree, but not on the values of the coefficients of $F$. This is crucial for our application.

\begin{proposition}\label{prop:sub_G}
  Let $n\geq 2$, $G$ a transitive subgroup of $S_n$ and $\ell\in\IN$. For $B\geq 2$, let $N_{n,G}(B)$ be the number of polynomials $f=a_0x^n+a_1x^{n-1}+\cdots+a_{n-1}x+a_n$ such that
  \begin{enumerate}
  \item $a_0,\ldots,a_{n}\in\IZ\cap [-B,B]$,
  \item $a_0,a_n$ are $\ell$-th powers in $\IZ\smallsetminus\{0\}$,
  \item $f$ is irreducible over $\IQ$,
  \item the Galois group of $f$ acts on the roots of $f$ (enumerated in a fixed order) as $G$.
    \end{enumerate}
  Then, for $\varepsilon>0$, we have the upper bound 
  \begin{equation}\label{eq:sub_G_bound}
N_{n,G}(B)\ll_{n,\varepsilon}B^{n-2+2/\ell+\#(S_n/G)^{-1}+\varepsilon}.
\end{equation}
\end{proposition}

\begin{proof}
  The result follows from Lemma \ref{countingprop} in case $n=2$, so we assume from now on that $n\geq 3$.
  Conditions (3) and (4) are invariant under replacing $f$ by 
  \begin{equation}\label{eq:g}
    a_0^{n-1}f(x/a_0)=x^n+a_1x^{n-1}+\cdots+a_0^{n-3}a_{n-2}x^2+a_0^{n-2}a_{n-1}x+a_0^{n-1}a_n,
  \end{equation}
  so we have to bound the number of $a_0,\ldots,a_n$ subject to (1) and (2), for which the polynomial in \eqref{eq:g} satisfies (3) and (4). Lemma \ref{lem:full_galois_Qt} shows that, for every choice of $a_0, a_2,\ldots,a_{n}$, there are $\ll_n 1$ choices of $a_1$ for which the polynomial
\begin{equation*}
 g(x;t)=x^n+a_1x^{n-1}+\cdots+a_0^{n-3}a_{n-2}x^2+tx+a_0^{n-1}a_n\in\IQ(t)[x]
\end{equation*}
does not have full Galois group $S_n$ over the rational function field $\IQ(t)$. The total number of $a_0,\ldots,a_n$ for which this holds is thus $\ll_{n}B^{n-2+2/\ell}$. In view of the desired bound \eqref{eq:sub_G_bound}, we may thus restrict our attention to those $a_0,\ldots,a_n$ for which
\begin{equation}\label{eq:assumpt_full_galois_group}
 g(x;t) \text{ has full Galois group $S_n$ over $\IQ(t)$}.
\end{equation}
 For these polynomials, we consider the corresponding Galois resolvents
 \begin{equation*}
\phi(z;t)=\phi(z;a_1,\ldots,a_0^{n-3}a_{n-2},t,a_0^{n-1}a_n)\in \IZ[z,t],
\end{equation*}
 defined in \eqref{eq:galois_resolvent}, and their discriminants $\Delta_\phi(t)=\Delta_\phi(a_1,\ldots,a_0^{n-3}a_{n-2},t,a_0^{n-1}a_n)\in\IZ[t]$.

Lemma \ref{lem:discriminant_non_zero} shows that, for any fixed permitted choice of $a_0, a_n$, the discriminant $\Delta_\phi(t)$ does not vanish identically as a polynomial in $a_1,\ldots,a_{n-2},t$. Hence, there are at most $\ll_n B^{n-3}$ choices of $a_1,\ldots,a_{n-2}$ with (1), for which $\Delta_\phi(t)=0$ in $\IQ(t)$. Summing this over all possible choices of $a_0,a_{n-1},a_n$ with (1) and (2), we obtain a contribution $\ll_n B^{n-2+2/\ell}$ in total, which is negligible when compared to \eqref{eq:sub_G_bound}. Hence, we may assume from now on that $\Delta_\phi(t)\neq 0$ for all our tuples $a_0,\ldots,a_n$ under consideration. In this case, together with our previous assumption \eqref{eq:assumpt_full_galois_group}, we see from Lemma \ref{lem:resolvent_irreducible} that $\phi(z,t)$ is irreducible over $\IQ$ for all choices of $a_0,\ldots,a_{n-2},a_n$. Fixing such a choice, suppose that the polynomial $g(x;a_0^{n-2}a_{n-1})$ from \eqref{eq:g} satisfies (3) and (4) for some $a_{n-1}$ subject to (1). 

Then all complex roots of $g(x;a_0^{n-2}a_{n-1})$ are distinct and moreover the Galois resolvent $\phi(z;a_0^{n-2}a_{n-1})$ has a root $z\in\IZ$. Since the roots of a complex polynomial are bounded polynomially in terms of its coefficients (see, e.g., \cite[Lemma 1]{MR2891158}), this root satisfies $|z|\leq B^{\alpha}$, for some $\alpha>0$ that depends at most on $n$. Since the polynomial $\phi(z;t)$, and thus also $\phi(z;a_0^{n-2}t)$, is irreducible over $\IQ$, we can apply Lemma \ref{lem:BHB} to bound the number of $(z,a_{n-1})\in\IZ^2$ with $|z|\leq P_1:=B^\alpha$ and $|a_{n-1}|\leq P_2:=B$ for which $\phi(z;a_0^{n-2}a_{n-1})=0$. Since the monomial $z^{\#(S_n/G)}$ appears in $\phi(z;t)$, we get $T\geq B^{\alpha\#(S_n/G)}$, and thus the number of such pairs $(z,a_{n-1})$ is
\begin{equation*}
  \ll_{n,\varepsilon} B^\varepsilon\exp\left(\frac{\alpha(\log B)^2}{\alpha\#(S_n/G)\log B}\right)=B^{\#(S_n/G)^{-1}+\varepsilon}.
\end{equation*}
Summing this over all viable choices of $a_0,\ldots,a_{n-2},a_n$ yields the bound \eqref{eq:sub_G_bound}.
\end{proof}

\begin{corollary}\label{cor:An_ext}
  Suppose $S\subset S_{\IQ,d}$ consists of all $A_d$-extensions and $\theta>d-3/2+2/\ell$. Then
\begin{alignat*}1
\Ndel(S,X)\ll_{d,\theta} X^{\theta}.
\end{alignat*}
\end{corollary}
\begin{proof}
  This is analogous to the proof of Lemma \ref{countingprop}, except
  that the relevant polynomials are now counted by Proposition \ref{prop:sub_G}
  instead of the trivial argument at the end of that proof.
  
  Let $P_S$ be the
  set of all $\alpha\in\Qbar$ such that $\IQ(\alpha)\in S$ and
  $\alpha\Oseen_{\IQ(\alpha)}=(\ppp_1\ppp_2^{-1})^{\ell}$, for prime ideals
  $\ppp_1\neq\ppp_2$ of $\Oseen_{\IQ(\alpha)}$ with $e(\ppp_i)=f(\ppp_i)=1$ for
  $i=1,2$. By Lemma \ref{lem:alpha_prop}, every field counted by $\Ndel(S,X)$ is of the
  form $\IQ(\alpha)$ for some $\alpha\in P_S$ with $H_{\IQ(\alpha)}(\alpha)\leq
  X$. By \eqref{minpoly} and the fact that $\IQ(\alpha)$ is an $A_d$-extension
  of $\IQ$, we see that the minimal polynomial $D_\alpha$ of $\alpha$ is
  counted by $N_{d,A_d}(2^dX)$. Propostion \ref{prop:sub_G} now shows that 
  \begin{equation*}
\Ndel(S,X)\ll_d N_{d,A_d}(2^dX)\ll_{d,\theta}X^{\theta}.
\end{equation*}
\end{proof}

\begin{corollary}\label{cor:D5_ext}
  Suppose $S\subset S_{\IQ,5}$ consists of all $D_5$-extensions and $\theta>3+1/12+2/\ell$. Then
\begin{alignat*}1
\Ndel(S,X)\ll_{\theta} X^{\theta}.
\end{alignat*}
\end{corollary}
\begin{proof}
  The proof is analogous to Corollary \ref{cor:An_ext}. Note that $\#(S_5/D_5)=12$.
\end{proof}

\begin{corollary}\label{cor:D4_ext}
  Suppose $S\subset S_{\IQ,4}$ consists of all $D_4$-extensions and $\theta>2+1/3+2/\ell$. Then
\begin{alignat*}1
\Ndel(S,X)\ll_{\theta} X^{\theta}.
\end{alignat*}
\end{corollary}

\begin{proof}
  Again, the proof is analogous to Corollary \ref{cor:An_ext}. Note that $\#(S_4/D_4)=3$.
\end{proof}

\section{Bounding the number of bad fields}\label{sec:badfields}
Recall that $d>1$ is an integer,
$S_{\IQ,d}=\{K\subset \Qbar\: ;\: [K:\IQ]=d\}$,
and for  $S\subset S_{\IQ,d}$
we defined $\B_S(X;Y,M)$ as the set
\begin{alignat*}1
\{K\in S\: ;\: X\leq \D< 2X, \text{ at most $M$ primes $p\leq Y$ split completely in $K$}\}.
\end{alignat*}

\begin{lemma}\label{lem:GRH}
  Let $d\geq 2$, and let $S\subset S_{\IQ,d}$ be a family of degree-$d$-fields. Suppose that the Riemann hypothesis holds for the Dedekind zeta function of the normal closure of each field in $S$. Then for every $\delta>0$ there exists $c=c(d,\delta)>0$ such that
  \begin{equation*}
    \#\B_S(X;X^\delta,c X^\delta/\log X) \ll_{d,\delta} 1.
  \end{equation*}
\end{lemma}

\begin{proof}
  This is an immediate consequence of the conditional effective version of Chebotarev's density theorem due to Lagarias and Odlyzko \cite{lo1977}.
\end{proof}

\begin{theorem}(\cite[Theorem 2.1]{EllenbergPierceWood})\label{exceptionalfields1}
Let $d\in \{3,4,5\}$, let $S=S_{\IQ,d}$ if $d\neq 4$ and $S=S^*_{\IQ,4}$ the family of all quartic non-$D_4$ fields, if $d=4$, and let  $\varepsilon>0$. Recall the definition of $\delta_0(d)$ (just before Theorem \ref{thm3}), and put
\begin{alignat*}1
\delta_0=\delta_0(d).
\end{alignat*}
Then for every $0<\delta\leq \delta_0$ there exists $c=c(\delta)>0$  such that 
\begin{alignat*}1
\#\B_S(X;X^{\delta},c X^{\delta}/\log X)\ll_{\delta, \varepsilon} X^{1-\delta+\varepsilon}.
\end{alignat*}
\end{theorem}

Consider families $S=S(G,\J)\subset S_{\IQ,d}$ of fields $K$ whose normal closure $\tilde{K}$  has 
Galois group $G$, and  such that for each rational prime $p$ that
is tamely ramified in $K$, its ramification is of type $\J$, where $\J$ specifies one or more conjugacy
classes in $G$. By this we mean the inertia group $I(\mathfrak{B})\subset G$ of any prime ideal $\mathfrak{B}\subset \Oseen_{\tilde{K}}$ above $p$ 
(which is cyclic if $p$ is tamely ramified in $K$) is generated by an element in the conjugacy class (or classes) specified by $\J$ (see \cite[\S 1.2.1]{PTW}). The following result collects some special cases of \cite[Corollary 3.16]{PTW}.

\begin{theorem}[Pierce, Turnage-Butterbaugh, Wood]\label{exceptionalfields2}
Let $\varepsilon>0$, let $S=S(G,\J)\subset S_{\IQ,d}$ be from one of the following five families, and let $\tau=\tau_S$ as below.
Then for every $\delta>0$ there exists $c=c(S,\delta)>0$ such that
\begin{alignat*}1
\#\B_S(X;X^{\delta},c X^{\delta}/\log X)\ll_{S,\delta,\cz, \tau,\varepsilon} X^{\tau+\varepsilon}.
\end{alignat*}
\begin{enumerate}
\item[1.] $G$ is a cyclic group of order $d\geq 2$ with $\J$ comprised of all generators of $G$ (equivalently,
every rational prime that is tamely ramified in $K$ is totally ramified), and $\tau=0$.
\item[2.] $d$ is an odd prime, and $G=D_d$ the Dihedral group of symmetries
  of a regular $d$-gon, with $\J$ being the conjugacy class of reflections and $\tau=1/(p-1)$. 
\item[3.] $d\geq 5$, $G=A_d$ and $\J = G$ (so no restriction on inertia type), and $\tau=0$. Moreover, assume that
the strong Artin conjecture holds for all irreducible Galois representations over $\IQ$ with image $A_d$.
\item[4.] $d\in \{3,4\}$, $G=S_d$, with $\J$ being the conjugacy class of transpositions, and $\tau=1/3$ if $d=3$ and $\tau=1/2$ if $d=4$.
\item[5.] $d\geq 5$, $G=S_d$, with $\J$ being the conjugacy class of transpositions, and the following two conditions hold: 
  \begin{enumerate}
  \item[(i)] the strong Artin conjecture holds for all
    irreducible Galois representations over $\IQ$ with image
    $S_d$,
  \item[(ii)] $\tau$ and $\cz$ are numbers such that $\tau<1$ if $d=5$ and $\tau<1/2+1/d$ if $d\geq 6$, and for every fixed integer $D$ there are at most $\cz D^{\tau}$
    fields $K\in S$ with $\D= D$. 
\end{enumerate}
\end{enumerate}
\end{theorem}

For the families $\Sab$ in Theorem \ref{thm:D4}, we have the following bounds, which follow from \cite[Theorem 1.6 and Proposition 6.1]{An}.

\begin{theorem}[An]\label{exceptionalfields3}
  Let $\varepsilon>0$, and let $a,b\in\IZ\smallsetminus\{0,1\}$ be distinct squarefree numbers. Then for every $\delta>0$ there exists $c=c(a,b,\delta)>0$ such that
\begin{alignat*}1
\#\B_{\Sab}(X;X^{\delta},c X^{\delta}/\log X)\ll_{a,b,\delta,\varepsilon} X^{\varepsilon}.
\end{alignat*}
\end{theorem}

\section{Proofs of theorems}\label{proofs}

Each of our Theorems follows immediately from one of the Propositions \ref{mainprop1} or \ref{mainprop2} with suitable parameters, combined with a simple application of dyadic summation.

\subsection{Proof of Theorem \ref{thm2}}
Apply Proposition \ref{mainprop2} with 
 $\theta=1+2/\ell$ (by Lemma \ref{countingprop}), $\rho=1$, and $\tau=0$ (by Theorem \ref{exceptionalfields2}).
 
\subsection{Proof of Theorem \ref{thm3}}
Apply Proposition \ref{mainprop1} with
 $\theta=d-1+2/\ell$ (by Lemma  \ref{countingprop}), $\rho=1$ (by \cite{81, Bhargava05, Bhargava10}), and $\delta_0=\delta_0(d)$ (by Theorem \ref{exceptionalfields1}).

\subsection{Proof of Theorem \ref{thm:GRH}}
Apply Propostion \ref{mainprop2} with $\theta=d-1+2/\ell$ (by Lemma \ref{countingprop}) and $\tau=0$ (by Lemma \ref{lem:GRH}).
 
\subsection{Proof of Theorem \ref{thm4}}
For sufficiently small $\varepsilon'>0$, we apply Proposition \ref{mainprop2} with $\theta=3/2+2/\ell+\varepsilon'$ (by Corollary \ref{cor:An_ext}), $\rho=1/2$ (by \cite{Wright89}), and $\tau=0$ (by Theorem \ref{exceptionalfields2}).

\subsection{Proof of Theorem \ref{thm:D5}}
For sufficiently small $\varepsilon'>0$, we apply Proposition \ref{mainprop2} with $\theta=3+1/12+2/\ell+\varepsilon'$ (by Corollary \ref{cor:D5_ext}) and $\tau=1/4$ (by Theorem \ref{exceptionalfields2}).

\subsection{Proof of Theorem \ref{thm:D4}}
For sufficiently small $\varepsilon'>0$, we apply Proposition \ref{mainprop2} with $\theta=2+1/3+2/\ell+\varepsilon'$ (by Corollary \ref{cor:D4_ext}) and $\tau=0$ (by Theorem \ref{exceptionalfields3}).

\subsection{Proof of Theorem \ref{thm5}}
First we note (cf. \cite[Lemma 6.9]{PTW}) that for each $S_d$-extension of degree $d$ with squarefree  discriminant, 
the  ramification type of each ramified prime $p$  that is tamely ramified is the conjugacy class of transpositions.
Now apply Proposition \ref{mainprop2} with $\theta=d-1+2/\ell$ (by Lemma \ref{countingprop}) and $\tau$ as in the statement of the theorem (by Theorem \ref{exceptionalfields2}).

\subsection{Proof of Theorem \ref{thm:Ad}}
For sufficiently small $\varepsilon'>0$, we apply Propostion \ref{mainprop2} with $\theta=d-3/2+2/\ell+\varepsilon'$ (by Corollary \ref{cor:An_ext}) and $\tau=0$ (by Theorem \ref{exceptionalfields2}).

\section{Upper bounds for Dihedral extensions}\label{Dhe}
The aim of this section is to prove Corollary \ref{cor1}. In the proof of \cite[Theorem 2.5]{Klueners}, Kl\"uners has shown the estimates 
\begin{alignat*}1
N(p,D_p,X)&\leq \sum_{\D^{(p-1)/2} b^{p-1}\leq X}\frac{p^{\omega(b)+r_K}-1}{p -1},\\
N(2p,D_{p}(2p),X)&\leq \sum_{\D^p b^{2(p-1)}\leq X}\frac{p^{\omega(b)+r_K}-1}{p -1},
\end{alignat*}
where both sums are taken over positive integers $b$ and  quadratic fields $K$ with $\D$ in the indicated range, $\omega(b)$ denotes the number of distinct prime divisors of $b$,
and $r_K$ is the $p$-rank of $\Cl_K$, so that  $p^{r_K}= \#\Cl_K[p]$. 
For the first sum we find
\begin{alignat*}1
N(p,D_p,X)\leq \sum_{\D^{(p-1)/2}b^{p-1}\leq X}\frac{p^{\omega(b)+r_K}-1}{p -1}\leq\sum_{b^{p-1}\leq X}p^{\omega(b)}
\sum_{\D\leq X^{2/(p-1)}/{b^{2}}}\#\Cl_K[p].
\end{alignat*}
Plugging in the bound from Theorem \ref{thm2} in case $k=1$ proves the claim for $N(p,D_p,X)$. The second sum is handled 
similarly.

\section*{Acknowledgments}
The authors are grateful to the referee for their careful reading and their valuable comments that significantly improved the exposition of the paper.

\bibliographystyle{alpha}
\bibliography{literature}

\end{document}